\documentclass[12pt]{article}
\usepackage{amsmath}
\usepackage{latexsym}
\usepackage{amssymb}
\newtheorem{thm}{Theorem}[section]
\newtheorem{la}[thm]{Lemma}
\newtheorem{Defn}[thm]{Definition}
\newtheorem{Remark}[thm]{Remark}
\newtheorem{prop}[thm]{Proposition}
\newtheorem{Number}[thm]{\!\!}
\newenvironment{defn}{\begin{Defn}\rm}{\end{Defn}}
\newenvironment{rem}{\begin{Remark}\rm}{\end{Remark}}
\newenvironment{numba}{\begin{Number}\rm}{\end{Number}}
\newenvironment{proof}{{\noindent\bf Proof.}}%
                  {\nopagebreak\hspace*{\fill}$\Box$\medskip\medskip\par}   
\newcommand{\Punkt}{\nopagebreak\hspace*{\fill}$\Box$}
\newcommand{\wb}{\overline}
\newcommand{\ve}{\varepsilon}
\newcommand{\at}{\symbol{'100}}
\newcommand{\wt}{\widetilde}
\newcommand{\mto}{\mapsto}
\newcommand{\isom}{\cong}
\DeclareMathOperator{\Ad}{Ad}
\newcommand{\N}{{\mathbb N}}
\newcommand{\R}{{\mathbb R}}
\newcommand{\K}{{\mathbb K}}

\newcommand{\C}{{\mathbb C}}
\newcommand{\cO}{{\cal O}}
\newcommand{\cg}{{\mathfrak g}}
\newcommand{\pl}{{\displaystyle \lim_{\longleftarrow}}}
\DeclareMathOperator{\Aut}{Aut}
\newcommand{\sub}{\subseteq}
\DeclareMathOperator{\im}{im}

\newcommand{\cL}{{\cal L}}
\DeclareMathOperator{\Lip}{Lip}
\DeclareMathOperator{\ev}{ev}
\DeclareMathOperator{\op}{op}
\DeclareMathOperator{\Evol}{Evol}
\DeclareMathOperator{\evol}{evol}
\DeclareMathOperator{\sym}{sym}
\begin{document}
$\;$\\[-27mm]
\begin{center}
{\Large\bf When unit groups of continuous inverse\\[3mm]
algebras are regular Lie groups}\\[7mm]
{\bf Helge Gl\"{o}ckner and Karl-Hermann Neeb}\vspace{4mm}
\end{center}
\begin{abstract}\noindent
It is a basic fact in infinite-dimensional Lie theory
that the unit group $A^\times$ of a continuous inverse
algebra $A$ is a Lie group.
We describe criteria ensuring that the Lie group $A^\times$
is regular in Milnor's sense. Notably, $A^\times$
is regular if $A$ is Mackey-complete and locally m-convex.
\vspace{1mm}
\end{abstract}
{\footnotesize {\em Classification}:
22E65 (primary);
%
32A12,
%
34G10,
%
46G20,
46H05,
%
58B10.\\[1mm]
%
}\\[6mm]
\noindent
{\bf Introduction and statement of the main result}\\[3mm]
A locally convex, unital, associative topological algebra $A$ over $\K\in \{\R,\C\}$
is called a \emph{continuous inverse algebra}
if its group $A^\times$ of invertible elements is open
and the inversion map $\iota\colon A^\times \to A$, $x\mto x^{-1}$
is continuous (cf.\ \cite{Wae}).
Then $\iota$ is $\K$-analytic
and hence $A^\times$ is a $\K$-analytic Lie group~\cite{ALG}.
Our goal is to describe conditions ensuring that the Lie group $A^\times$
is well-behaved, i.e., a regular Lie group in the sense of Milnor~\cite{Mil}.\\[2.5mm]
To recall this notion,
let $G$ be a Lie group modelled on a locally convex space~$E$,
with identity element~$1$, its tangent bundle $TG$
and the Lie algebra $\cg:=T_1G\isom E$.
Given $g\in G$ and $v\in T_1G$,
let $\lambda_g\colon G\to G$, $x\mto gx$ be left translation
by~$g$ and $gv:=T_1(\lambda_g)(v)\in T_gG$.
If $\gamma\colon [0,1]\to\cg$ is a continuous map,
then there exists at most one
$C^1$-map $\eta\colon [0,1]\to G$
such that
\[
\eta'(t)=\eta(t)\gamma(t)\quad\mbox{for all $\,t\in [0,1]$, \,and\, $\eta(0)=1$.}
\]
If such an $\eta$ exists, it is called the \emph{evolution of $\gamma$}.
The Lie group $G$ is called \emph{regular}
if each $\gamma\in C^\infty([0,1],\cg)$ admits an evolution $\eta_\gamma$,
and the map $\evol\colon C^\infty([0,1],\cg)\to G$, $\gamma\mto \eta_\gamma(1)$
is smooth (see \cite{Mil} and \cite{Nee},
where also many applications of regularity are described).
If $G$ is regular, then its modelling space~$E$ is \emph{Mackey-complete} in the sense
that every Lipschitz curve in $E$ admits a Riemann integral\footnote{See
\cite{KaM} for a detailed discussion of this property.}
(as shown in \cite{GaN}). It is a notorious open problem whether, conversely, every Lie group
modelled on a Mackey-complete locally convex space is regular
(\cite[Problem II.2]{Nee}; cf.\,\cite{Mil}).\\[2.7mm]
As a tool for the discussion of $A^\times$,
we let $\mu_n\colon A^n\to A$ be the $n$-linear map
defined via $\mu_n(x_1,\ldots, x_n):=x_1x_2\cdots x_n$,
for $n\in \N$.
Given seminorms $p,q\colon A\to [0,\infty[$, we define $\wb{B}^q_1(0):=\{x\in A\colon q(x)\leq 1\}$
and
\[
\|\mu_n\|_{p,q}:=\sup\{p(\mu_n(x_1,\ldots, x_n))\colon x_1,\ldots, x_n\in \wb{B}^q_1(0)\}
\in [0,\infty]\,.
\]
Our regularity criterion now reads as follows:\\[4mm]
{\bf Theorem.}
\emph{Let $A$ be a Mackey-complete continuous inverse algebra,
such that the following condition is satisfied}:
\begin{itemize}
\item[($*$)]
\emph{For each continuous seminorm~$p$ on~$A$,
there exists a continuous seminorm~$q$ on~$A$
and $r>0$
$($which may depend on $p$ and $q)$
such that}
\[
\sum_{n=1}^\infty r^n\|\mu_n\|_{p,q}\; <\; \infty\, .
\]
\end{itemize}
\emph{Then $A^\times$ is a regular Lie group in Milnor's sense.}\\[4mm]
$A^\times$ even satisfies certain stronger
regularity properties (see Proposition~\ref{stronger}).
Of course, by Hadamard's formula for the radius of convergence of a power series,
condition ($*$) is equivalent to\footnote{If $\|\mu_n\|_{p,q}<\infty$, then also $\|\mu_k\|_{p,q}<\infty$
for all $k\in \{1,\ldots, n\}$.
In fact, $\|\mu_k\|_{p,q}\leq q(1)^{n-k}\|\mu_n\|_{p,q}$
since $\mu_k(x_1,\ldots, x_k)=\mu_n(1,\ldots,1,x_1,\ldots, x_k)$.}
$\, {\displaystyle\limsup_{n\to\infty}}\sqrt[n]{\|\mu_n\|_{p,q}}<\infty$.
It is also equivalent to the existence of $M\in [0,\infty[$
such that $\|\mu_n\|_{p,q}\leq M^n$ for all $n\in \N$.\\[2mm]
{\bf Remark A.} The authors do not know whether condition ($*$)
can be\linebreak
omitted, i.e., whether $A^\times$ is regular for
every Mackey-complete continuous inverse algebra~$A$.
Here are some preliminary considerations:\\[2.5mm]
If $A$ is a continuous inverse algebra, then the map $\pi_n\colon A\to A$, $x\mto x^n$
is a continuous homogeneous polynomial of degree $n$,
for each $n\in \N_0$.
It is known that the analytic inversion map $\iota\colon A^\times\to A$
is given by Neumann's series, $\iota(1-x)=\sum_{n=0}^\infty x^n=\sum_{n=0}^\infty \pi_n(x)$,
for $x$ in some $0$-neighbourhood of $A$ \cite[Lemma~3.3]{ALG}.
Hence, for each continuous seminorm~$p$ on~$A$, there exists a continuous seminorm~$q$ on~$A$ and $s>0$
such that\vspace{-.3mm}
\[
\sum_{n=1}^\infty s^n\|\pi_n\|_{p,q}\; <\; \infty\,,\vspace{-.3mm}
\]
where $\|\pi_n\|_{p,q}:=\sup\{p(\pi_n(x))\colon x\in \wb{B}^q_1(0)\}$
(cf.\ \cite[Proposition~5.1]{BaS}).\footnote{If $\K=\R$,
we can apply the proposition to $A_\C$,
which is a complex continuous inverse algebra
(see, e.g., \cite[Proposition~3.4]{ALG}).}
Let $S_n$ be the symmetric group of all permutations of $\{1,\ldots, n\}$
and $\mu_n^{\sym}\colon A^n\to A$, $(x_1,\ldots, x_n)\mto \frac{1}{n!}\sum_{\sigma\in S_n}x_{\sigma(1)}\cdots x_{\sigma(n)}$
be the symmetrization of $\mu_n$. Then $\pi_n(x)=\mu_n^{\sym}(x,\ldots,x)$
and thus $\|\mu_n^{\sym}\|_{p,q}\leq \frac{n^n}{n!}\|\pi_n\|_{p,q}$
by the Polarization Formula
(in the form \cite[p.\,34,\,(2)]{Her}).
Since $\lim_{n\to\infty}\frac{n}{\sqrt[n]{n!}}=e$
is Euler's constant
(as a consequence of Stirling's Formula),
it follows that\vspace{-.3mm}
\begin{equation}\label{forabel}
\sum_{n=1}^\infty \, t^n\|\mu^{\sym}_n\|_{p,q}\;<\; \infty\quad\mbox{for each $\,t\in\, ]0,\frac{s}{e}[$.}\vspace{-.3mm}
\end{equation}
In general, it is not clear how one could give good estimates
for $\|\mu_n\|_{p,q}$ in terms of $\|\mu_n^{\sym}\|_{p,q}$.
Hence, it does not seem to be clear in general whether
(\ref{forabel}) implies the existence of some $r>0$ with ($*$).\\[3mm]
However, ($*$) is satisfied in important cases.
Following \cite{EAM}, a topological algebra $A$ is called
\emph{locally m-convex} if its topology arises from a set of seminorms $q$
which are \emph{submultiplicative}, i.e., $q(xy)\leq q(x)q(y)$
for all $x,y\in A$.\\[3.6mm]
{\bf Corollary.} \emph{Let $A$ be a Mackey-complete continuous inverse
algebra. If $A$ is commutative or locally
m-convex,
then $A^\times$ is a regular Lie group.}\\[4mm]
\begin{proof}
In fact, if $A$ is commutative, then $\mu_n=\mu_n^{\sym}$,
whence ($*$) is satisfied with any $r\in \,]0,s/e[$ as in (\ref{forabel}).
Therefore the theorem applies.\footnote{Alternative proof:
$(A,+)$ is regular, as it is Mackey complete \cite[Proposition~II.5.6]{Nee}.
Since $\exp\colon A\to A^\times$ is a homomorphism
of groups (as $A^\times$ is abelian) and a local diffeomorphism (see \cite[Theorem~5.6]{ALG}),
it follows that also $A^\times$ is regular \cite[Proposition~3]{Rob}.}\\[2.5mm]
If $A$ is locally m-convex, then
for every continuous seminorm~$p$ on~$A$, there is a submultiplicative
continuous seminorm $q$ on~$A$ such that $p\leq q$.
Using the submultiplicativity,
we see that $\|\mu_n\|_{p,q}\leq \|\mu_n\|_{q,q}\leq 1$.
Thus ($*$) is satisfied with any $r\in \,]0,1[$,
and the theorem applies.
\end{proof}
\noindent
It can be shown that every Mackey-complete, commutative continuous inverse
algebra is locally m-convex (cf.\ \cite{Tur}).\vspace{3mm}

\noindent
{\bf Remark~B.}
We mention that there is a quite direct, alternative proof for the corollary
if $A$ is locally m-convex and \emph{complete}.\footnote{Then $A=\pl \,A_q$\vspace{-.4mm} is a projective limit of Banach algebras
(where $q$ ranges through the set of all submultiplicative continuous seminorms on~$A$).
Being a Banach-Lie group, each $A_q^\times$ is regular~\cite{Mil}.
Then $C^\infty([0,1],A)=\pl \, C^\infty([0,1],A_q)$\vspace{-.7mm} and $\evol_{A^\times}=\pl \evol_{A_q^\times}$
is a smooth evolution (cf.\ \cite[Lemma~10.3]{BGN}).}
The easier arguments fail however if~$A$ is not complete,
but merely sequentially complete or Mackey-complete.
By contrast, our more elaborate method does not require that~$A$ be complete:
Mackey-completeness suffices.\\[2.7mm]
{\bf Remark C.}
Our theorem is a variant of the (possibly too optimistic)\linebreak
Theorem~IV.1.11 announced in the survey \cite{Nee},
and its proof expands the sketch of proof given there.
To avoid the difficulties described in Remark~A,
we added
condition~($*$).\\[2.5mm]
{\bf Remark~D.}
It is known that
unit groups of Mackey-complete
continuous inverse algebras are so-called \emph{BCH-Lie groups} \cite[Theorem~5.6]{ALG},
i.e., they admit an exponential function which is an analytic
diffeomorphism around~$0$ (see \cite{GCX}, \cite{Nee}, \cite{Rob}
for information on such groups).
Inspiration for the studies came from an article
by Robart~\cite{Rob}.
He pursued the (possibly too optimistic) larger goal
to show that every BCH-Lie group with Mackey-complete
modelling space is regular. However, there seem to be gaps
in his arguments.\footnote{For example,
it is unclear whether the limit $\gamma_u$
constructed in the proof of \cite[Proposition~7]{Rob}
takes its values in $\Aut(\cL)$
(as observed by K.-H. Neeb),
and no explanation is given how a smooth curve $g$
in the local group with $\Ad(g)=\gamma_u$ can be obtained.}\\[4mm]
{\bf Remark~E.}
The following questions are open:
\begin{itemize}
\item[(a)]
Are there examples of Mackey-complete continuous inverse
algebras which satisfy ($*$) but are not locally m-convex? Or even:
\item[(b)]
Does every Mackey-complete continuous inverse
algebra satisfy~($*$)\,?
\end{itemize}
\section{Notation and preparatory results}\label{secone}
\noindent
In this section, we fix some notation and formulate
preparatory results.\\[2,7mm]
{\bf Basic notation.}
Let $\N=\{1,2,\ldots\}$ and $\N_0:=\N\cup\{0\}$.
If $X$ is a set and $n\in \N$, we write
$X^n:=X\times\cdots\times X$ (with $n$ factors).
If $f\colon X\to Y$ is a map,
we abbreviate $f^n:=f\times\cdots\times f\colon X^n\to Y^n$, $(x_1,\ldots, x_n)\mto (f(x_1),\ldots, f(x_n))$.\linebreak
If $(E,\|.\|_E)$ and $(F,\|.\|_F)$ are normed spaces
and $\beta\colon E^n\to F$ is a\linebreak
continuous
$n$-linear map, we write $\|\beta\|_{\op}$ for its operator norm, defined as usual as
$\sup\{\|\beta(x_1,\ldots, x_n)\|_F\colon x_1,\ldots, x_k\in E, \|x_1\|_E,\ldots, \|x_n\|_E\leq 1\}$.
If $E$ is a locally convex space, we let $P(E)$ be the set of all continuous
seminorms on~$E$. If $p\in P(E)$, we consider the factor space $E_p:=E/p^{-1}(0)$
as a normed space with the norm $\|.\|_p$ given by $\|x+p^{-1}(0)\|_p:=p(x)$.
Then the canonical map $\pi_p\colon E\to E_p$, $x\mto x+p^{-1}(0)$ is linear and continuous,
with $\|\pi_p(x)\|_p=p(x)$.\\[2.7mm]
{\bf Weak integrals.}
Recall that if $E$ is a locally convex space, $a\leq b$ are reals
and $\gamma\colon [a,b]\to E$ a continuous map,
then the weak integral $\int_a^b\gamma(s)\,ds$
(if it exists)
is the unique element of $E$ such that $\lambda(\int_a^b\gamma(s)\,ds)=\int_a^b\lambda(\gamma(s))\,ds$
for each continuous linear functional $\lambda$ on~$E$.
If $\alpha\colon E\to F$ is a continuous linear map between locally convex spaces
and $\int_a^b\gamma(s)\, ds$ (as before) exists in~$E$,
then also $\int_a^b\alpha(\gamma(s))\,ds$ exists in~$F$ and is given by
\begin{equation}\label{linweakint}
\int_a^b\alpha(\gamma(s))\,ds=\alpha\left(
\int_a^b\gamma(s)\,ds\right)
\end{equation}
(see, e.g., \cite{GaN} for this observation).
If $E$ is sequentially complete, then $\int_a^b\gamma(s)\,ds$ always exists (cf.\ \cite[Lemma~1.1]{BaS}
or \cite[1.2.3]{Her}).\\[3mm]
{\bf\boldmath$C^r$-curves.}
Let $r\in \N_0\cup\{\infty\}$.
As usual, a \emph{$C^r$-curve} in a locally convex space~$E$ is a
continuous function $\gamma\colon I\to E$
such that the derivatives $\gamma^{(k)}\colon I\to E$ of order $k$
exist for all $k\in \N$ with $k\leq r$,
and are continuous (see, e.g., \cite{GaN} for more details).
The $C^\infty$-curves are also called \emph{smooth curves}.\\[3mm]
{\bf Smooth maps.}
If $E$ and $F$ are real locally convex spaces, $U\sub E$ is an open subset
and $r\in \N_0\cup\{\infty\}$,
then a function $f\colon U\to F$ is called $C^r$ if is continuous,
the iterated directional derivatives $d^{(k)}f(x,y_1,\ldots,y_k):=(D_{y_k}\cdots D_{y_1}f)(x)$
exist for all $k\in \N$ such that $k\leq r$, $x\in U$ and $y_1,\ldots, y_k\in E$,
and define continuous functions $d^{(k)}f\colon U\times E^k\to F$.
If $U$ is not open, but a convex (or locally convex) subset of~$E$ with dense
interior~$U^0$, we say that $f$ is $C^r$ if it is continuous,
$f|_{U^0}$ is $C^r$ and $d^{(k)}(f|_{U^0})\colon U^0\times E^k\to F$
has a continuous extension $d^{(k)}f\colon U\times E^k\to F$ for each $k\in \N$ such that $k\leq r$.
The $C^\infty$-maps are also called \emph{smooth}.
Abbreviate $df:=d^{(1)}f$.
It is known that the Chain Rule holds in the form $d(f\circ g)(x,y)=df(g(x),dg(x,y))$,
and that compositions of $C^r$-maps are $C^r$.
Moreover, a $C^0$-curve $\gamma\colon I\to E$ is a $C^r$-curve
if and only if it is a $C^r$-map, in which case $\gamma'(t)=d\gamma(t,1)$
(see \cite{GaN} for all of these basic facts; cf.\ also \cite{Mic}, \cite{Mil},
and \cite{RES}).\\[3mm]
{\bf Analytic maps.}
If $E$ and $F$ are complex locally convex spaces
and $n\in \N$, then a function $p\colon E\to F$
is called a \emph{continuous homogeneous polynomial}
of degree $n\in \N_0$ if $p(x)=\beta(x,\ldots,x)$
for some continuous $n$-linear map $\beta\colon E^n\to F$
(if $n=0$, this means a constant function).
A map $f\colon U\to F$ on an open set $U\sub E$
is called \emph{complex analytic} (or $\C$-analytic) if it is continuous and
for each
$x\in U$, there is a $0$-neighbourhood $Y\sub E$ with $x+Y\sub U$
and continuous homogeneous
polynomials $p_n\colon E\to F$ of degree~$n$,
such that
\[
(\forall y\in Y)\quad f(x+y)=\sum_{n=0}^\infty p_n(y)
\]
(see \cite{BaS}, \cite{RES} and \cite{GaN} for further information).
Following \cite{Mil}, \cite{RES} and \cite{GaN} (but deviating from \cite{BaS}),
given real locally convex spaces $E, F$,
we call a function $f\colon U\to F$ on an open set $U\sub E$ \emph{real analytic}
(or $\R$-analytic) if it extends to a complex analytic map $V\to F_\C$,
defined on some open subset $V\sub E_\C$ of the complexification
of~$E$, such that $U\sub V$.
For $\K\in \{\R,\C\}$,
it is known that compositions of $\K$-analytic maps are $\K$-analytic.
Every $\K$-analytic map is smooth (see, e.g.,  \cite{GaN} or \cite{RES}
for both of these facts).\\[2.7mm]
We shall use the following lemma (proved in Appendix~A):
\begin{la}\label{special}
Let $E$ and $F$ be complex locally convex spaces, $\wt{F}$ be a completion of $F$
such that $F\sub \wt{F}$ as a dense vector subspace,
and $p_n\colon E\to F$ be continuous homogeneous polynomials of degree~$n$,
for $n\in \N_0$. Assume that
\[
f(x):=\sum_{n\in \N_0}p_n(x)\vspace{-2mm}
\]
converges in $\wt{F}$ for all $x$ in a balanced, open $0$-neighbourhood
$U\sub E$, and $f\colon U\to\wt{F}$ is continuous.
If $F$ is Mackey-complete, then $f(x)\in F$
for all $x\in U$ and $f\colon U\to F$ is $\C$-analytic.
\end{la}
{\bf Function spaces.}
If $E$ is a locally convex space and $r\in \N_0\cup\{\infty\}$,
let $C^r([0,1],E)$ be the space of all $C^r$-maps from $[0,1]$ to $E$.
We endow $C^r([0,1],E)$ with the locally convex vector topology defined
by the seminorms $\|.\|_{C^k,p}$ given by
\[
\|\gamma\|_{C^k,p}:=\max_{j=0,\ldots,k}\max_{t\in [0,1]}p(\gamma^{(j)}(t)),
\]
for $p$ in the set of continuous seminorms on~$E$
and $k\in \N_0$ with $k\leq r$.
We abbreviate $C([0,1],E):=C^0([0,1],E)$.
Three folklore lemmas concerning these function spaces will be used.
(The proofs can be found in Appendix~A).
\begin{la}\label{embfuncs}
Let $E$ and $F$ be locally convex spaces
$\alpha\colon E\to F$
be a continuous linear map, and $r\in \N_0\cup\{\infty\}$.
Then also the map
\[
\alpha_*:=C^r([0,1],\alpha)\colon C^r([0,1],E)\to C^r([0,1],F), \quad
\gamma\mto\alpha\circ\gamma
\]
is continuous and linear.
If $\alpha$ is a
topological embedding $($i.e., a homeomorphism onto its image$)$,
then also $\alpha_*$ is a topological embedding.
\end{la}
\begin{la}\label{cpro}
If $E$ is a locally convex space and $r\in \N_0\cup\{\infty\}$,
then the\linebreak
topology on $C^r([0,1],E)$ is initial with respect to the mappings
$(\pi_p)_*:$\linebreak
$C^r([0,1],E)\to C^r([0,1],E_p)$, $\gamma\mto \pi_p\circ \gamma$
for $p\in P(E)$.
\end{la}
\begin{la}\label{funmac}
\hspace*{2mm}If $r\in \N_0\cup\{\infty\}$ and $E$ is a locally convex space $E$ which is
complete $($resp., Mackey-complete$)$,
then also
$C^r([0,1],E)$ is complete $($resp., Mackey-complete$)$.
\end{la}
\section{Picard iteration of paths in a topological algebra}
\begin{numba}\label{setting}
Let $A$ be a locally convex topological algebra over $\C$,
i.e., a unital, associative, complex algebra,
equipped with a Hausdorff locally convex vector topology making
the map $A\times A\to A$, $(x,y)\mto xy$ continuous.
We assume that condition ($*$) from Theorem~A is satisfied.\footnote{Note that $A$
is not assumed to be a continuous inverse algebra in this section.}
\end{numba}
If $E$ is a locally convex space, then a function $\gamma\colon [0,1]\to E$
is a \emph{Lipschitz curve} if $\{\frac{\gamma(t)-\gamma(s)}{t-s}\colon s\not=t\in[0,1]\}$
is bounded in~$E$ (cf.\ \cite[p.\,9]{KaM}). For our current purposes, we endow the space
$\Lip([0,1],E)$
of all such curves with the topology $\cO_{C^0}$ induced by $C^0([0,1],E)$.
\begin{la}[Picard Iteration]\label{noinverse}
Let $A$ be as in {\rm\ref{setting}}.
If $A$ is sequentially complete
and $\gamma\in C([0,1],A)$,
we can define a sequence $(\eta_n)_{n\in \N}$
in $C^1([0,1],A)$ via $\eta_0(t):=1$,
\[
\eta_n(t)=1+\int_0^t\eta_{n-1}(t_n)\gamma(t_n)\,dt_n\quad\mbox{for $\,t\in [0,1]$ and $\,n\in \N$.}
\]
Then the following holds:
\begin{itemize}
\item[\rm(a)] The limit
$\eta:=\eta_\gamma:={\displaystyle \lim_{n\to\infty}}\eta_n$ exists in $C^1([0,1],A)$.
\item[\rm(b)]
$\eta_n(t)=1+\sum_{k=1}^n\int_0^t\int_0^{t_{k-1}}\cdots\int_0^{t_2}\gamma(t_1)\cdots\gamma(t_k)\,dt_1\ldots dt_k$
for each $n\in \N_0$ and $t\in [0,1]$, and thus
\begin{equation}\label{sumform}
\eta(t)=1+\sum_{n=1}^\infty\int_0^t\int_0^{t_{n-1}}\!\!\!\!\!\cdots\int_0^{t_2}\gamma(t_1)\cdots\gamma(t_n)\,dt_1\ldots dt_n.
\end{equation}
\item[\rm(c)]
$\eta'(t)=\eta(t)\gamma(t)$ and $\eta(0)=1$.
\item[\rm(d)]
The map $\Phi\colon C([0,1],A)\to C^1([0,1],A)$, $\gamma\mto\eta_\gamma$
is complex analytic.
\end{itemize}
If $A$ is not sequentially complete, but
Mackey-complete, then the $(\eta_n)_{n\in \N_0}$
can be defined and {\rm(a)--(c)} hold
for each $\gamma\in \Lip([0,1],A)$. Moreover,
\begin{itemize}
\item[\rm(d)$'$]
$\Phi\colon (\Lip([0,1],A), \cO_{C^0}) \to C^1([0,1],A)$, $\gamma\mto\eta_\gamma$
is complex analytic.
\end{itemize}
\end{la}
\begin{proof}
If $A$ is sequentially complete, set
$X:=C([0,1],A)$; otherwise,
set $X:=\Lip([0,1],A)$.
Let $\wt{A}$ be a completion of $A$ such that $A\sub\wt{A}$.
Then the inclusion map
$\phi\colon C^1([0,1],A)\to C^1([0,1],\wt{A})$
is a topological embedding (Lemma~\ref{embfuncs})
and $C^1([0,1],\wt{A})$ is complete (Lemma~\ref{funmac}).
Hence also the closure $Y\sub C^1([0,1],\wt{A})$ of the image $\im(\phi)$
is complete, and thus $\phi\colon C^1([0,1],A)\to Y$
is a completion of $C^1([0,1],A)$.

Proof of (a) and (b), (d) and (d)$'$:
Let $\gamma\in X$.
Then all integrals needed to define
$\sigma_n(\gamma):=\eta_n$ exist, and each $\eta_n$ is $C^1$,
by the Fundamental Theorem of Calculus.
A trivial induction shows that
\begin{equation}\label{altforme}
\eta_n(t)=1+\sum_{k=1}^n\int_0^t\int_0^{t_{k-1}}\!\!\!\cdots\int_0^{t_2}\gamma(t_1)\cdots\gamma(t_k)\,dt_1\ldots dt_k
\end{equation}
(as asserted in (b)).
Likewise, if $n\in  \N$ and $\gamma_1,\ldots,\gamma_n \in X$,
then the weak integrals needed to define $\tau_n(\gamma_1,\ldots,\gamma_n)\colon [0,1] \to  A$,
\[
t\mto
\int_0^t\int_0^{t_n}\cdots\int_0^{t_2} \gamma_1(t_1)\cdots \gamma_n(t_n)\,dt_1\cdots dt_n
\]
exist and $\tau_n(\gamma_1,\ldots,\gamma_n)$
is a $C^1$-map.
Since $\tau_n\colon C([0,1],A)^n\to C^1([0,1],A)$,
$(\gamma_1,\ldots,\gamma_n)\mto \tau_n(\gamma_1,\ldots,\gamma_n)$
is $n$-linear and $\sigma_n(\gamma)=\tau_n(\gamma,\ldots,\gamma)$,
the map $\sigma_n\colon C([0,1],A)\to C^1([0,1],A)$
is a homogeneous polynomial of degree~$n$
(and this conclusion also holds if $n=0$,
as $\sigma_0$ is constant).
If $p\in P(A)$, there is $q\in P(A)$ and $M\in [0,\infty[$ such that
\[
(\forall n\in \N)\quad \|\mu_n\|_{p,q}\leq M^n\,,
\]
as a consequence of condition ($*$).
Applying $p$ to the iterated integral in (\ref{altforme}), we deduce
that
$p(\sigma_n(\gamma)(t))\leq
\frac{t^n}{n!} \|\mu_n\|_{p,q}
\|\gamma\|_{C_0,q}^n\leq
\frac{t^nM^n}{n!}
\|\gamma\|_{C_0,q}^n$ for each $t\in [0,1]$
and thus
\begin{equation}\label{niceest1}
\|\sigma_n(\gamma)\|_{C^0,p}\leq
\frac{M^n}{n!}
\|\gamma\|_{C^0,q}^n.
\end{equation}
Also, $\sigma_1(\gamma)'(t)=\gamma(t)$
and
\begin{equation}\label{pre}
\sigma_n(\gamma)'(t)=
\int_0^t\int_0^{t_{n-1}}\!\!\!\cdots\int_0^{t_2} \gamma(t_1)\cdots\gamma(t_{n-1})\gamma(t)\,dt_1\cdots dt_{n-1}
\end{equation}
if $n\geq 2$, by the Fundamental Theorem of Calculus.
Thus
\begin{equation}\label{prediffa}
(\forall n\in \N)\quad \eta_n'(t)=\eta_{n-1}(t)\gamma(t)\,,
\end{equation}
which will be useful later.
By~(\ref{pre}), also
$p(\sigma_n(\gamma)'(t))\leq\frac{t^{n-1}}{(n-1)!}
\|\mu_n\|_{p,q}
\|\gamma\|_{C^0,q}^n$
and thus
\begin{equation}\label{niceest2}
\|\sigma_n(\gamma)'\|_{C^0,p}\leq\frac{M^n}{(n-1)!}
\|\gamma\|_{C^0,q}^n.
\end{equation}
Combining (\ref{niceest1}) and (\ref{niceest2}), we see that
\begin{equation}\label{niceest3}
\|\sigma_n(\gamma)\|_{C^1,p}\leq
\frac{M^n}{(n-1)!} \|\gamma\|_{C^0,q}^n\,.
\end{equation}
Therefore $\sigma_n\colon C([0,1],A)\to C^1([0,1],A)$
is a continuous homogeneous\linebreak
polynomial.
Moreover, we obtain
\[
\sum_{n=1}^\infty\|\sigma_n(\gamma)\|_{C^1,p}
\leq\sum_{n=1}^\infty
\frac{M^n\|\gamma\|_{C^0,q}^n}{(n-1)!}=M\|\gamma\|_{C^0,q}e^{M\|\gamma\|_{C^0,q}}<\infty\,.
\]
This estimate entails that the series $\sum_{n=0}^\infty\sigma_n(\gamma)$
converges absolutely in the\linebreak
completion $Y$ of $C^1([0,1],A)$.
In particular, the limit
\[
\Phi(\gamma):=\sum_{n=0}^\infty\sigma_n(\gamma)
\]
exists in~$Y$, and defines a function
$\Phi\colon X\to Y$.
We claim that~$\Phi$ is continuous.
If this is true, then we
can exploit that $C^1([0,1],A)$ is Mackey-complete by Lemma~\ref{funmac},
and each $\sigma_n$ takes its values inside $C^1([0,1],A)$.
Thus all hypotheses of Lemma~\ref{special} are satisfied,
and we deduce that
$\Phi(\gamma)\in C^1([0,1],A)$ for each $\gamma$
(entailing (a) and (b)),
and that the map $\Phi\colon X\to C^1([0,1],A)$ is complex analytic
(establishing (d) and (d)$'$, respectively).\\[2.7mm]
To establish the claim, we need only show that
$\Phi$ is continuous
as a map to $C^1([0,1],\wt{A})$. Identify $p\in P(A)$ with its continuous
extension to a seminorm on~$\wt{A}$. Let $\pi_p\colon \wt{A}\to((\wt{A})_p,\|.\|_p)$
be the canonical map. By Lemma~\ref{cpro}, $\Phi$ will be continuous if the maps
$h:=(\pi_p)_*\circ \Phi\colon X\to C^1([0,1],(\wt{A})_p)$ are continuous.
It suffices that~$h$ is continuous
on the ball $B_R:=\{\gamma\in X\colon \|\gamma\|_{C^0,q}<R\}$
for each $R>0$.
However,
\[
h(\gamma)=\sum_{n=0}^\infty \pi_p\circ \sigma_n(\gamma)
\]
for $\gamma\in B_R$, where
\[
\|\pi_p\circ \sigma_n(\gamma)\|_{C^1,\|.\|_p}
=\|\sigma_n(\gamma)\|_{C^1,p}\leq
\frac{M^n}{(n-1)!} \|\gamma\|_{C^0,q}^n
\leq
\frac{M^n}{(n-1)!} R^n
\]
for $n\in\N$,
by (\ref{niceest3}).
Hence
\[
\sum_{n=0}^\infty \sup\{\pi_p\circ \sigma_n(\gamma)\colon \gamma\in B_R\}
\leq 1+ MRe^{RM}<\infty\,,
\]
entailing that $\sum_{k=0}^n((\pi_p)_*\circ\sigma_n|_{B_R})\to h|_R$
uniformly. Thus $h|_{B_R}$ is continuous, being a uniform
limit of continuous functions.

(c) Because $\eta_n\to \eta$ in $C^1([0,1],A)$, we know that $\eta_n'\to \eta'$ uniformly
(and thus pointwise).
Letting $n\to\infty$ in (\ref{prediffa}), we deduce that $\eta'(t)=\eta(t)\gamma(t)$.
\end{proof}
\section{Proof of the theorem}
We establish our theorem
as a special case of a more general result\linebreak
(Proposition~\ref{stronger}).
The latter deals with certain strengthened
regularity\linebreak
properties (as used earlier in~\cite{SUR}
and \cite{Da2}):
\begin{defn}
Let $G$ be a Lie group modelled on a locally convex space,
with Lie algebra $\cg$, and $k\in \N_0\cup\{\infty\}$.
\begin{itemize}
\item[(a)]
$G$ is called \emph{strongly $C^k$-regular}
if every $\gamma\in C^k([0,1],\cg)$ admits an evolution $\Evol(\gamma)\in C^1([0,1],G)$
and the map $\evol\colon C^k([0,1],\cg)\to G$, $\gamma\mto\Evol(\gamma)(1)$ is smooth.
\item[(b)]
$G$ is called \emph{$C^k$-regular}
if each $\gamma\in C^\infty([0,1],\cg)$ has an evolution
and the map $\evol\colon (C^\infty([0,1],\cg),\cO_{C^k})\to G$, $\gamma\mto\Evol(\gamma)(1)$ is smooth,
where $\cO_{C^k}$ denotes the topology induced by $C^k([0,1],\cg)$ on $C^\infty([0,1],\cg)$.
\end{itemize}
\end{defn}
The reader is referred to \cite{GDL}
for a detailled discussion of these
regularity properties (and applications depending thereon).
Both $C^\infty$-regularity and strong $C^\infty$-regularity
coincide with regularity in the usual sense.
If $k\leq \ell$ and $G$ is (strongly) $C^k$-regular,
then $G$ is also (strongly) $C^\ell$-regular.
\begin{rem}\label{ofcourse}
If $A$ is a continuous inverse algebra, we identify the tangent bundle $T(A^\times)$ of
the open set~$A^\times$ with $A^\times \times A$
in the natural way. Let $\eta\colon [0,1]\to A^\times$ be a $C^1$-curve
and $\gamma\colon [0,1]\to A$ be continuous.
Then $\eta'(t)=\eta(t)\gamma(t)$ holds in $T(A^\times)$
(using $\eta'\colon [0,1]\to T(A^\times)$, and identifying the range $A$ of $\gamma$
with $\{1\}\times A\sub T_1(A^\times)$)
if and only if $\eta'(t)=\eta(t)\gamma(t)$ holds in $A$ (where the product
simply refers to the algebra multiplication, and $\eta'\colon [0,1]\to A$
is the derivative of the $A$-valued $C^1$-curve $\eta$).
\end{rem}
The next lemma will help us to see that the
$A$-valued map $\eta$ associated to~$\gamma$ in Lemma~\ref{noinverse} actually
takes its values in~$A^\times$, if $A$ is a continuous inverse algebra.
Hence $\eta$ will be the evolution of $\gamma$, by Remark~\ref{ofcourse}.
\begin{la}\label{ridnoninv}
Let $A$ be a continuous inverse algebra, $\gamma\colon [0,1]\to A$ be continuous
and $\eta\colon [0,1]\to A$ as well as $\zeta\colon[0,1]\to A$
be $C^1$-curves. Assume that  $\eta(0)=\zeta(0)=1$ and
\begin{equation}\label{multhy}
\eta'(t)=\eta(t)\gamma(t)\quad \mbox{and}\quad
\zeta'(t)=\zeta(t)\gamma(t)\quad\mbox{for all $t\in [0,1]$.}
\end{equation}
If $\zeta([0,1])\sub A^\times$, then $\eta=\zeta$.
\end{la}
\begin{proof}
Recall from \cite[proof of Lemma~3.1]{ALG}
that the differential of the inversion map $\iota\colon A^\times\to A$ is given by
$d\iota(a,b)=-a^{-1}ba^{-1}$ for $a\in A^\times$ and $b\in A$.
As a consequence, the derivative of the $C^1$-curve $\iota\circ \zeta\colon [0,1]\to A^\times$, $t\mto \zeta(t)^{-1}$
is given by
\begin{equation}\label{present}
(\iota\circ\zeta)'(t)=-\zeta(t)^{-1}\zeta'(t)\zeta(t)^{-1}.
\end{equation}
Now consider the $C^1$-curve $\theta\colon [0,1]\to A$, $\theta(t):=\eta(t)\zeta(t)^{-1}$.
Using the Product Rule, (\ref{present}) and (\ref{multhy}), we obtain
\begin{eqnarray*}
\theta'(t)&=&\eta'(t)\zeta(t)^{-1}-\eta(t)\zeta(t)^{-1}\zeta'(t)\zeta(t)^{-1}\\
&=&\eta(t)\gamma(t) \zeta(t)^{-1}-\eta(t)\zeta(t)^{-1}\zeta(t)\gamma(t)\zeta(t)^{-1}\\
&=&\eta(t)\gamma(t) \zeta(t)^{-1}-\eta(t)\gamma(t)\zeta(t)^{-1}\,=\, 0\,.
\end{eqnarray*}
Hence $\theta(t)=\theta(0)=\eta(0)\zeta(0)^{-1}=1$ for all $t\in [0,1]$
and thus $\eta=\zeta$.
\end{proof}
\begin{prop}\label{stronger}
Let $A$ be a continuous inverse
algebra over $\K\in \{\R,\C\}$,
which satisfies the condition {\rm($*$)}
described in the introduction.
\begin{itemize}
\item[\rm(a)]
If $A$ is sequentially complete,
then $A$ is strongly $C^0$-regular
and the map $\Evol\colon C^0([0,1],A)\to C^1([0,1],A^\times)$
is $\K$-analytic.
\item[\rm(b)]
If $A$ is Mackey-complete,
then $A$ is $C^0$-regular and strongly
$C^1$-regular.
Further, each $\gamma\!\in\! \Lip([0,1],A)$ has
an evolution $\Evol(\gamma)\!\in\! C^1([0,1],A^\times)$,
and $\Evol  \colon (\Lip([0,1],A),\cO_{C^0})\to C^1([0,1],A^\times)$
is $\K$-analytic.
\end{itemize}
\end{prop}
\begin{proof}
If $A$ is sequentially complete, let $X:=C([0,1],A)$;
otherwise, let $X:=(\Lip([0,1],A),\cO_{C^0})$.\\[2.5mm]
We assume first that $\K=\C$.
Let $\Phi\colon X\to C^1([0,1],A)$
be the mapping provided by Lemma~\ref{noinverse}.
Note that $C^1([0,1],A^\times)\sub C^1([0,1],A)$ is an identity neighbourhood,
$\Phi(0)=1$ (cf.\ (\ref{sumform}))
and $\Phi$ is $\C$-analytic (see (d) or (d)$'$)
and hence continuous.
Therefore, there exists an open $0$-neighbourhood $\Omega\sub X$ such that
$\Phi(\Omega)\sub C^1([0,1],A^\times)$.
By Proposition~\ref{noinverse}\,(c), $\Evol(\gamma):=\Phi(\gamma)$ is an evolution
for $\gamma\in \Omega$. Moreover,
$\evol\colon \Omega\to A^\times$, $\gamma\mto \Evol(\gamma)(1)=\Phi(\gamma)(1)$
is $\C$-analytic, since $\Phi$ and the continuous linear
point evaluation $\ev_1\colon C^1([0,1],A)\to A$, $\zeta\mto\zeta(1)$ are
$\C$-analytic.\\[2.5mm]
If $A$ is sequentially complete,
Proposition 1.3.10 in \cite{Da2} now shows that $A^\times$ is strongly $C^0$-regular.\footnote{Compare
already \cite[p.\,409]{KaM} and
\cite[Lemma~3]{Rob} for similar arguments.}\\[2.5mm]
If $A$ is Mackey-complete,
we see as in the proof of \cite[Proposition~1.3.10]{Da2}
that each $\gamma\in \Lip([0,1],A)$
has an evolution $\Evol(\gamma)\in C^1([0,1],A^\times)$.\\[2.5mm]
In either case, we deduce with Lemmas \ref{noinverse}\,(c) and \ref{ridnoninv}
that $\Evol=\Phi$. As a consequence,
$\Evol\colon X\to C^1([0,1],A^\times)$
is $\C$-analytic. Thus (a) is established.
In the situation of (b),
note that also $\evol:=\ev_1\circ \Evol\colon \Lip([0,1],A) \to A^\times$
is $\C$-analytic.
The inclusion maps $(C^\infty([0,1],A),\cO_{C^0})\to (\Lip([0,1],A),\cO_{C^0})$
and $C^1([0,1],A)\to (\Lip([0,1],A),\cO_{C^0})$ being continuous linear and hence
$\C$-analytic, it follows that also the maps
$\evol\colon (C^\infty([0,1],A),\cO_{C^0})\to A^\times$
and
$\evol\colon C^1([0,1],A)\to A^\times$ are $\C$-analytic
and thus smooth.
Hence $A^\times$ is $C^0$-regular and strongly $C^1$-regular.\\[2.5mm]
If $\K=\R$, then also the complexification $A_\C$
of $A$ is a continuous inverse algebra (see, e.g., \cite[Proposition~3.4]{ALG})
with the same completeness properties.
In (a), we can identify $X_\C$ with $C^0([0,1],A_\C)$;
in the situation of~(b), we can identify $X_\C$ with $\Lip([0,1],A_\C)$.
For $p\in P(A)$, let $p_\C\in P(A_\C)$ be the seminorm defined via
\[
p_\C(a+ib):=\inf\big\{{\textstyle \sum_j}\, |z_j| p(x_j) \colon a+ib={\textstyle \sum_j}\, z_jx_j,\; x_j\in A, z_j\in \C\big\}
\]
for $a,b\in A$ (which satisfies $\max\{p(a),p(b)\}\leq p_\C(a+ib)\leq p(a)+p(b)$).
Then also $A_\C$ satisfies ($*$),
as $\|(\mu_n)_\C\|_{p_\C,q_\C}=\|\mu_n\|_{p,q}$.
Let $\Phi\colon X_\C\to C^1([0,1],A_\C)$
be the complex analytic map provided by Lemma~\ref{noinverse}
(applied to $A_\C$ in place of~$A$).
By the complex case just discussed,
\[
\Phi=\Evol_{(A_\C)^\times}\colon X_\C\to C^1([0,1],(A_\C)^\times).
\]
If $\gamma\in X$,
then $\Phi(\gamma)$ takes only values in the closed vector subspace~$A$
of $A_\C=A\oplus i A$, as is clear from~(\ref{sumform}).
Hence $\Phi(\gamma)\in C^1([0,1],A)$ (see \cite{GaN} or \cite[Lemma~10.1]{BGN})
and thus $\Phi(\gamma)\in C^1([0,1],A^\times)$,
using the fact that $A\cap (A_\C)^\times= A^\times$
for any unital algebra.\footnote{If $x,a,b\in A$ and $x(a+ib)=(a+ib)x=1$, then $xa+ixb=1$ and $ax+ibx=1$.
Hence $xa=ax=1$, i.e., $x^{-1}=a\in A$.}
We deduce that $\Phi|_X\colon X\to C^1([0,1],A^\times)$ is the evolution map
$\Evol_{A^\times}$ of $A^\times$. Note that $\Evol_{A^\times}$ is $\R$-analytic,
because $\Phi\colon X_\C\to C^1([0,1],A)_\C$ is a $\C$-analytic extension
of $\Evol_{A^\times}$. Since evaluation $\ev_1\colon C^1([0,1],A)\to A$,
$\zeta\mto\zeta(1)$ is continuous linear and hence $\R$-analytic,
also $\evol_{A^\times}:=\ev_1\circ \Evol_{A^\times}\colon X\to A^\times$
is $\R$-analytic (and hence smooth). In the situation
of (a), this completes the proof.
In~(b), compose $\evol_{A^\times}$
with the continuous linear inclusion maps
$C^1([0,1],A)\to \Lip([0,1],A)$ (resp., $(C^\infty([0,1],A),\cO_{C^0})\to \Lip([0,1],A)$)
to see that also the evolution map on $C^1([0,1],A)$ (resp., on
$(C^\infty([0,1],A^\times),\cO_{C^0})$) is $\R$-analytic and hence~$C^\infty$.
\end{proof}
\appendix
\section{Proofs for the lemmas from Section~\ref{secone}}
It is useful to recall that a locally convex space~$E$ is Mackey-complete
(in the sense recalled in the introduction)
if and only if every Mackey-Cauchy sequence in~$E$ converges,
i.e., every sequence $(x_n)_{n\in \N}$ in $E$ for which there exists a bounded
subset $B\sub E$ and a double sequence $(r_{n,m})_{n,m\in \N}$
of real numbers $r_{n,m}\geq 0$ such that $x_n-x_m\in r_{n,m}B$
for all $n,m\in \N$, and $r_{n,m}\to 0$ as both $n,m\to\infty$
(cf.\ \cite[Theorem~2.14]{KaM}).\\[3mm]
{\bf Proof of Lemma~\ref{special}.}
Given $x\in U$, there exists $r\in \,]1,\infty[$
such that $rx\in U$. Thus $\sum_{n=0}^\infty r^np_n(x)$ converges
and hence $C:=\{r^np_n(x)\colon n\in \N_0\}$ is a bounded
subset of~$F$. Then also the absolutely convex hull $B$ of $C$ is bounded.
For all $n,m\in \N_0$,
we have
\begin{eqnarray*}
\sum_{k=0}^{n+m}p_k(x)-\sum_{k=0}^np_k(x) &= & \sum_{k=n+1}^{n+m}p_k(x)
\,=\, r^{-n-1}\sum_{k=n+1}^{n+m} r^{n+1-k}r^kp_k(x)\\
&\in&  r^{-n-1}\left(\sum_{j=0}^{m-1}(1/r)^j\right)B\sub\frac{r^{-n-1}}{1-\frac{1}{r}}B.
\end{eqnarray*}
Hence $(\sum_{k=0}^np_k(x))_{n\in \N_0}$
is a Mackey-Cauchy sequence in $F$ and hence convergent.
Thus $f(x)\in F$. Note that $f$ is complex
analytic as a map to $\wt{F}$, by \cite[Theorems~5.1 and 6.1\,(i)]{BaS}.
Hence, if $x\in U$, then $f(x+y)=\sum_{n=0}^\infty\frac{1}{n!}\delta^n_x(f)(y)$
for all $y$ in some $0$-neighbourhood,
where $\delta^n_xf(y):=d^{(n)}f(x,y,\ldots,y)$ is the $n$-th G\^{a}teaux differential of
$f$ at~$x$.
Given $y\in E$, there is $s>0$ such that $x+zy\in U$ for all $z\in \C$ such that $|z|\leq s$.
For each $n\in \N_0$,
Cauchy's Integral Formula for higher derivatives now shows that 
\[
\delta^n_x(f)(y)=\frac{n!}{2\pi i}\int_0^{2\pi}\frac{f(x+se^{it}y)}{(se^{it})^{n+1}}\, sie^{it}dt,
\]
which lies in~$F$ since the integrand is a Lipschitz curve in~$F$
and~$F$ is Mackey-complete.\footnote{The integrand
is a $C^\infty$-curve in $\wt{F}$ and hence a Lipschitz curve in $\wt{F}$,
with image in~$F$.}
Hence each $\delta^n_x(f)$ is a continuous homogeneous
polynomial from $E$ to $F$ and thus $f$ is complex
analytic as a map from~$E$ to~$F$.\vspace{3mm}\Punkt

\noindent
{\bf Proof of Lemma~\ref{embfuncs}.}
Let $p$ be a continuous seminorm on~$F$ and $k\in \N_0$
such that $k\leq r$. Then $q:=p\circ \alpha$ is a continuous seminorm on~$E$.
Let $\gamma\in C^r([0,1],E)$.
For each $j\in \N_0$ such that $j\leq k$,
we have $(\alpha\circ \gamma)^{(j)}=\alpha\circ \gamma^{(j)}$
and thus $\|(\alpha\circ\gamma)^{(j)}\|_{C^0,p}
=\|\alpha\circ\gamma^{(j)}\|_{C^0,p}
=\|\gamma^{(j)}\|_{C^0,p\circ\alpha}
=\|\gamma^{(j)}\|_{C^0,q}$,
entailing that $\|\alpha\circ \gamma\|_{C^k,p}=\|\gamma\|_{C^k,q}$.
Hence $\alpha_*$ is continuous.

If $\alpha$ is an embedding and $Q$ a continuous seminorm on $C^r([0,1],E)$,
then there exists $k\in \N_0$ such that $k\leq r$ and a continuous seminorm~$q$
on~$E$ such that $Q\leq \|.\|_{C^k,q}$.
Because $\alpha$ is an embedding, there exists a continuous seminorm~$p$ on~$F$
such that $p(\alpha(x))\geq q(x)$ for all $x\in E$ (because $\alpha^{-1}$ is continuous linear).
Hence $\|(\alpha\circ\gamma)^{(j)}\|_{C^0,p}
=\|\gamma^{(j)}\|_{C^0,p\circ \alpha}\geq \|\gamma^{(j)}\|_{C^0,q}$
for each $j\in \N_0$ such that $j\leq k$ and thus
$\|\alpha\circ \gamma\|_{C^k,p}\geq \|\gamma\|_{C^k,q}\geq Q(\gamma)$,
entailing that $\alpha_*$ is a topological embedding.\vspace{3mm}\Punkt

\noindent
{\bf Proof of Lemma~\ref{cpro}.}
Let $p\in P(E)$ and $k\in \N_0$ such that $k\leq r$.
Since $p=\|.\|_p\circ\pi_p$, we have
\[
\|(\pi_p\circ \gamma)^{(j)}\|_{C^0,\|.\|_p}=
\|\pi_p\circ \gamma^{(j)}\|_{C^0,\|.\|_p}=
\|\gamma^{(j)}\|_{C^0,\|.\|_p\circ \pi_p}
=\|\gamma^{(j)}\|_{C^0,p}
\]
for each $\gamma\in C^r([0,1],E)$ and $j\in \{0,1,\ldots, k\}$,
entailing that $\|\pi_p(\gamma)\|_{C^k,\|.\|_p}=\|\gamma\|_{C^k,p}$.
The assertion follows.\Punkt

\begin{rem}\label{redorder}
Before we turn to the proof of Lemma~\ref{funmac},
it is useful to record some simple observations.
\begin{itemize}
\item[(a)]
It is clear from the definitions that the map
\[
h\colon C^k([0,1],E)\to C([0,1],E)\times C^{k-1}([0,1],E)
\,,\quad \gamma\mto (\gamma,\gamma')
\]
is linear and a homeomorphism onto its image, for each $k\in \N$.
\item[(b)]
The image $\im(h)$ of $h$ consists of all $(\gamma,\eta)$ such that $\gamma(t)=\gamma(0)+\int_0^t\eta(s)\,ds$
for each $t\in [0,1]$. Since point evaluations and the linear maps $\eta\mto\int_0^t\eta(s)\, ds$
(with $p(\int_0^t\eta(s)\, ds)\leq \|\eta\|_{C^0,p}$) are continuous, it follows that
$\im(h)$ is a closed vector subspace of $C([0,1],E)\times C^{k-1}([0,1],E)$.
\end{itemize}
\end{rem}
\noindent
{\bf Proof of Lemma~\ref{funmac}.}
Because direct products of Mackey-complete\linebreak
locally convex spaces are Mackey-complete,
and so are closed vector\linebreak
subspaces, also
projective limits of Mackey-complete locally convex spaces are Mackey-complete.
Since $C^\infty([0,1],E)=\pl\,C^k([0,1],E)$\vspace{-.7mm}
(with the\linebreak
respective
inclusion maps as the limit maps), we therefore
only need to prove Mackey-completeness if $k:=r\in \N_0$.
Likewise in the case of completeness.\vspace{.5mm}

The case $k=0$. If $E$ is complete, then also $C([0,1],E)$ is complete, as is well known
(cf.\ \cite[Chapter~7, Theorem~10]{Kel}).
If $E$ is merely Mackey-complete,
let $\wt{E}$ be a completion of $E$ which contains $E$.
Then $C([0,1],\wt{E})$ is complete.
The inclusion map $\phi\colon C([0,1],E)\to C([0,1],\wt{E})$ is a topological
embedding, by Lemma~\ref{embfuncs}.
If $(\gamma_n)_{n\in \N}$ is a Mackey-Cauchy sequence in $C([0,1],E)$,
then $(\phi\circ\gamma_n)_{n\in \N}=(\gamma_n)_{n\in \N}$
is a Mackey-Cauchy sequence in $C([0,1],\wt{E})$,
hence convergent to some $\gamma\in C([0,1],\wt{E})$.
For each $t\in [0,1]$, the point evaluation $\ve_t\colon C([0,1],\wt{E})\to \wt{E}$,
$\eta\mto\eta(t)$ is continuous and linear.
Hence $(\gamma_n(t))_{n\in \N}$ is
a Mackey-Cauchy sequence in $E$ and hence convergent
in~$E$. Since $\gamma_n(t)=\ve_t(\gamma_n)\to\ve_t(\gamma)=\gamma(t)$,
we deduce that $\gamma(t)\in E$.
Hence $\gamma\in C([0,1],E)$ and it is clear that $\gamma_n\to \gamma$ in $C([0,1],E)$.

Induction: If $C^{k-1}([0,1],E)$ is (Mackey-) complete, then also $C^k([0,1],E)$,
as it is isomorphic to a closed vector subspace of the (Mackey-) complete
direct product $C([0,1],E)\times C^{k-1}([0,1],E)$ (see Remark~\ref{redorder}\,(b)).\vspace{3mm}\Punkt

\noindent
\emph{Acknowledgement.}
The research was supported by DFG, grant GL 357/5-1.\vspace{-2mm}
{\small
Helge Gl\"{o}ckner, Universit\"{a}t  Paderborn, Institut f\"ur Mathematik,\\
Warburger Str.\ 100, 33098 Paderborn, Germany;\,{\tt  glockner\at{}math.upb.de}\\[2.5mm]
%
%
Karl-Hermann Neeb, FAU Erlangen-N\"{u}rnberg, Department Mathematik,\\
Cauerstr.\ 11, 91058 Erlangen, Germany;\\
{\tt karl-hermann.neeb\at{}math.uni-erlangen.de}}
\end{document}